\documentclass{amsart}

\usepackage{graphicx} 
\usepackage{pstricks-add}
\usepackage{pstricks,pst-node,pst-coil,pst-plot,pst-text}
\usepackage{pstricks}
\usepackage[all,cmtip]{xy}
\usepackage{color}
\allowdisplaybreaks

\newtheorem{theorem}{Theorem}[section]
\newtheorem{lemma}[theorem]{Lemma}

\theoremstyle{definition}

\newtheorem{example}[theorem]{Example}

\theoremstyle{remark}
\newtheorem{remark}[theorem]{Remark}

\numberwithin{equation}{section}

\newcommand{\eps}{\epsilon} 
\newcommand{\R}{{\mathbb{R}}} 
\newcommand{\N}{{\mathbb{N}}} 
\newcommand{\ds}{\displaystyle} 

\begin{document}

\title{Small- and large-amplitude limit cycles in Kukles systems with algebraic invariant curves}

\author{Jos\'e Mujica}
\address{Department of  Mathematics, Vrije Universiteit Amsterdam, The Netherlands}
\email{j.p.mujica@vu.nl}


\date{January 2022.}

\dedicatory{}


\begin{abstract}
Limit cycles of planar polynomial vector fields have been an active area of research for decades; the interest in periodic-orbit related dynamics comes from Hilbert's 16th problem and the fact that oscillatory states are often found in applications. We study the existence of limit cycles and their coexistence with invariant algebraic curves in two families of Kukles systems, via Lyapunov quantities and Melnikov functions of first and second order. We show center conditions, as well as a connection between small- and large-amplitude limit cycles arising in one of the families, in which the first coefficients of the Melnikov function correspond to the first Lyapunov quantities. We also provide an example of a planar polynomial system in which the cyclicity is not fully controlled by the first nonzero Melnikov function.   
\end{abstract}

\maketitle

\pagestyle{myheadings}
\thispagestyle{plain}
\markboth{ J. Mujica}
{Small- and large-amplitude limit cycles in Kukles systems}

\section{Introduction}
We consider planar systems of the form 
\begin{equation}\label{eq:Ku}
\begin{array}{rcl}
\dot x &=&-y\\[0,3cm]
\dot y &=&f(x,y),
\end{array}
\end{equation}
where $f(x,y)$ is a polynomial function with real coefficients and not divisible by~$y$. Systems like  \eqref{eq:Ku} are known as \emph{Kukles} systems, named after \emph{Isaak Solomonovich Kukles} (1905-1977), who showed necessary and sufficient center conditions for a cubic system of the form 
\begin{equation}\label{eq:Ku2}
\left\{
\begin{array}{rcl}
\dot x &=&-y,\\[0,3cm]
\dot y &=&x+a_1x^2+a_2xy+a_3y^2+a_4x^3+a_5x^2y+a_6xy^2+a_7y^3,
\end{array}
\right.
\end{equation}
where $a_1,\ldots , a_7$ are real coefficients. Finding necessary and sufficient center conditions for a planar system means to find conditions on the parameters so that the system has an equilibrium point (usually the origin) with a Jacobian matrix having complex eigenvalues with zero real part and all the orbits in a neighborhood of the equilibrium point are topologically equivalent to a circle. Finding such conditions, specially the second one, is far from being straight since a center is a codimension-infinity phenomenom that depends highly on the system's nature. The conditions given by Kukles for system \eqref{eq:Ku} were:
$$
\begin{array}{rcl}
k_{\alpha}=k_{\beta}=k_{\gamma}=a_7=0&&(K1),\\[0,3cm]
a_7=a_2=a_5=0&&(K2),\\[0,3cm]
a_7=a_5=a_3=a_1=0&&(K3),\\[0,3cm]
k_{\alpha}=k_{\beta}=k_{\gamma}=k_{\delta}=0&&(K4),
\end{array}
$$
with
$$
\begin{array}{rcl}
k_{\alpha}&=&a_4a_2^2+a_5\lambda\\[0,3cm]
k_{\beta}&=&(3a_7\lambda+\lambda^2+a_6a_2^2)a_5-3a_7\lambda^2-a_6a_2^2\lambda\\[0,3cm]
k_{\gamma}&=&\lambda+a_1a_2+a_5\\[0,3cm]
k_{\delta}&=&9a_6a_2^2+2a_4^2+9\lambda^2+27a_7\lambda
\end{array}
$$
and
$$
\lambda=a_2a_3+3a_7.
$$
Conditions $(K1)$--$(K4)$ are known as Kukles conditions \cite{Kukles}. 

Since Kukles' work  more researchers became interested in Kukles systems, due to its connection with Hilbert's 16th problem and the fact that Kukles systems are often found in applications via Li\'enard equations; they can model several oscillatory phenomena \cite{Car,Kukles}. In particular, these systems appear in models for heart rythm and breathing patterns \cite{VdP, VdP-VdM}, neuronal activity \cite{FH,Li2}, chemical reactions \cite{Field}, microtube kinetics and self organization on cellular level \cite{Glade}, as well as dynamics of populations of neurons, genetic models and biological regulation \cite{Thomas,Wilson}.  

Using computational techniques, Jin y Wang \cite{Jin-Wang} computed the focal quantities of \eqref{eq:Ku2} and noticed that there was a condition missed by Kukles, which provided a weak focus of much higher order than expected for these systems. Such conditions were
$$
a_2=a_6=0; \ a_3=-2a_1;\ a_5=-3a_7;\ a_7^2=a_4^2;\ -a_4=\dfrac {a_1^2}{3}.
$$
In \cite{Chri}, Cristopher and Lloyd proved that Jin and Wang's conditiones were indeed center conditions, claiming that Kukles conditions where incomplete. They classify all center conditions when $a_7=0$, verifying the case completely studied by Kukles. They also proved that for $a_7=0$ the maximum order of a weak focus is five. In~\cite{Chri2} the case $a_7\ne 0$ is studied, together with the center conditions not considered previously by Kukles, Jin and Wang.

As a remark, it is known that there exists a one-to-one correspondence between Kukles systems and second-order ordinary differential equations of the form $$\ddot x +f(x,\dot x)=0,$$ where $f(x,y)$ is a polinomial function. In addition, all singularities of a Kukles system lie on the $x$-axis.

\subsection{Limit cycles in Kukles systems}
One of the most relevant questions in planar polynomial systems and, in particular for Kukles systems, is about the existence and number of limit cycles, together to their stability and configuration on phase plane. A limit cycle is an isolated periodic orbit, and it is one of the responsibles for the organization of oscillatory patterns. The problem of finding limit cycles also takes special relevance due to its connection with Hilbert's 16th problem \cite{Chris-Lloyd, Du-El-Ro, Du-Ro-Ro, Hil, Ily, Il-Ya, Rous}.

In order to obtain limit cycles in \eqref{eq:Ku}, it necessary that $f(x,y)$ is not divisible by $y$ in \eqref{eq:Ku}. If $y$ divides $f(x,y)$, we can write \eqref{eq:Ku} as

\begin{equation}\label{eq:Ku3}
\left\{
\begin{array}{rcl}
\dot x &=&-y\\[0,3cm]
\dot y &=&yg(x,y),
\end{array}
\right.
\end{equation}
where $g(x,y)$ is a polynomial function with real coefficients. Here, all points on the straight line $y=0$ are singularities. It is known that any limit cycle of \eqref{eq:Ku3} must enclose a singularity, and since all singularities lie on the $x$-axis, any limit cycle must cross the $x$-axis and contain a singularity. This contradicts the uniqueness of solutions for system \eqref{eq:Ku3}. Therefore, system \eqref{eq:Ku3} does not have limit cycles. In what follows, we consider \eqref{eq:Ku} with $f(x,y)$ not divisible by $y$.
\bigskip

The number and configuration of limit cycles for Kukles systems have been widely studied. In \cite{Sad}, Sadovskii proved via Hopf bifurcations that \eqref{eq:Ku2} has at least 7 small-amplitude limit cycles that bifurcate from a weak focus. In \cite{Gaiko}, the authors showed that the number of limit cycles in \eqref{eq:Ku2} is at least 3, while in \cite{Ye-Ye} it is shown that \eqref{eq:Ku2} has at least 6 limit cycles enclosing the origin.  In \cite{Zang-Tian-Tade} the authors study a reduced Kukles system with a cubic perturbation:

$$
\left\{
\begin{array}{rcl}
\dot x &=& -y\\[0,3cm]
\dot y &=& x-\dfrac{a+1}{a}x^2+\dfrac{2-a}{1-a}y^2+\dfrac{1}{a}x^3+\eps g(x,y)
\end{array}
\right.
$$

with $a>2$, $\eps>0$ and $g$  given by:

$$
g(x,y)=b_{01}y+b_{11}xy+b_{02}y^2+b_{12}xy^2+b_{21}x^2y+b_{03}y^3
$$

and $a_{ij}, b_{ij} \in \mathbb R$.\\

The authors obtained 3 different configurations for 5 limit cycles in the system, with 2 of them having global (large-amplitude) limit cycles. In \cite{Cha-Sa-Sza-Grau},
 the authors study all possible configurations of invariant straight lines in Kukles systems, together with providing bounds for the number of limit cycles. They also 
give necessary conditions for the existence of an invariant algebraic curve of degree greater or equal to 2 that coexists with limit cycles. S\'aez and  Sz\'ant\'o 
\cite{Sa-Sz3} found a family of Kukles systems of arbitrary degree having an invariant non-degenerate ellipse $h(x,y)=0$:
$$
\left\{
\begin{array}{rcl}
\dot x &=& -y\\[0,3cm]
\dot y &=& -dc+bx+q_{n-2}(x,y)f(x,y)
\end{array}
\right.
$$

with $h(x,y)=d-2dcx+bx^2+y^2,\ b>0,\ d>0$ and for $0\le e <1,\ a>0$,

$$
\begin{array}{rcl}
b&=&1-e^2\\[0,3cm]
c&=&-\dfrac{1}{a}\\[0,3cm]
d&=&-e^2a^2
\end{array}
$$

and

$$
q_{n-2}(x,y)=-\dfrac{1}{a}+\ds\sum_{i=1}^{n-2}q_{i0}x^i+q_{0i}y^i.
$$
One can rewrite $h(x,y)=x^2+y^2-e^2(x+a)^2$, where $e$ corresponds to the excentricity of the ellipse.
 They proved that the ellipse is an invariant algebraic limit
 cycle. They also found bounds for the number of limit cycles and showed that the coefficients of the first Melnikov function coincide with the Lyapunov quantities. There are also studies on modified Kukles systems, e.g. generalized Kukles systems \cite{Lli-Mereu} and extended Kukles systems \cite{Hill-Lloyd-Pear}, to name a few.\\
 
In this paper, we study the existence of limit cycles for systems of the form~\eqref{eq:Ku}, and their coexistence with invariant algebraic curves. Recall that an algebraic curve $C$ is invariant with respect to a system such as \eqref{eq:Ku} if there exists a polynomial $K$ such that $\dot C = CK$; here, $\dot C$ denotes the rate of change along the orbits. We consider two leading examples: we first consider a family of Kukles systems of degree four having an invariant circle. We show center conditions and also prove the existence and uniqueness of a limit cycle via Lyapunov quantities and Melnikov functions, it turns out that the unique limit cycle is the invariant circle; the novelty here is the use Melnikov functions of second order. Secondly, we consider a family of Kukles systems of arbitrary odd degree with an invariant circle. We show an upper bound for the number of limit cycles of large amplitude and provide a concrete example when there is a connection between the small- and large-amplitude limit cycles via a relation between the Lyapunov quantities and the coefficients of the first Melnikov function. Here the total number of limit cycles is not totally controled by the Melnikov function and there is an interplay between limit cycles bifurcating from a weak focus and a center.
\\

This paper is organized a follows. Sections \ref{sec:center} and \ref{sec:Melnikov} provide a review of the techniques and methods used in this work. Even though they are not extensive and far from being a complete survey of methods for studying limit cycles in planar poltnomial systems, they intend to be a help for non-expert readers, so that they can go through the rest of the paper. In particular, section \ref{sec:center} introduces Lyapunov quantities and show how they can be used for studying limit cycles that bifurcate form a weak focus. Also, section \ref{sec:Melnikov} describes the situation when the vector field is a perturbation of a Hamitonian system and introduce Melnikov functions to study limit cycles that bifurcate from a center; these are the main tools used throughout this paper. We provide the corresponding references for the results exposed there. Sections \ref{sec:kukcir} and \ref{sec:odd} present the main results. In section \ref{sec:kukcir} we study the coexistence of limit cycles with an invariant circle in a Kukles family of degree four. Finally, in  section \ref{sec:odd} we study a more general case, when limit cycles coexist with an invariant circle in a Kukles family of arbitrary odd degree.

\section{Center conditions and limit cycles}
\label{sec:center}
By using Lyapunov functions \cite{Hahn,Wig} one can determine the stability of a singularity in a planar system. In this section, we will see that the good choice of a Lyapunov function also gives information on the number of limit cycles in planar polynomial systems and can be used to provide center conditions.

In what follows we consider a planar polynomial vector field $X$ with an isolated singularity at the origin. We assume that that the origin is a center-focus of $X$, that is, the Jacobian matrix at the origin $DX(0,0)$ has eigenvalues with zero real part. A normal form for such system is given by
\begin{equation}\label{eq:FN}
X=\left\{
\begin{array}{llll}
\dot{x} & = & P(x,y) = \lambda x-y+p(x,y) \\
 & \\

\dot{y} & = & Q(x,y) = x+\lambda y+q(x,y)
\end{array}
\right.
\end{equation}

where $p,q$ are polynomials without linear terms and $\lambda \in
\R$.\\

For the origin to be a center of \eqref{eq:FN}, it is necessary that $\lambda=0$.  If $\lambda=0$ and the origin is not a center,
we call it a \emph{weak focus}. The necessary conditions for a center can be obtained by calculating the \emph{focal quantities}, which are polynomials in the 
coeficients of \eqref{eq:FN} that are computed using a Lyapunov function of the form
\begin{equation}\label{eq:Lyap}
V(x,y)=\ds\frac12(x^2+y^2)+V_3(x,y)+V_4(x,y)+\ldots
,
\end{equation}

where $V_j(x,y)$ is a homogeneous polynomial of degree $j$ such that
$$\dot V=\eta_2r^2+\eta_4r^4+\ldots+\eta_{2k}r^{2k}+\ldots$$

with $r=x^2+y^2$.
\medskip

To find $\eta_{2k}, k\in \mathbb N$, one solves a sequence of linear equations in the coefficients of $P$ and $Q$ as follows. If we denote by $D_k$ the coefficients of degree $k$ in $\dot V$,
 the condition 
$D_k=0$ provides a system of linear equations that allows to find the coefficients of $V_k$. This linear system has a unique solution for $k$ even. In
this case one can find $\eta_{2k}$ by setting 
$D_{2k}=\eta_{2k}(x^2+y^2)^k$. For more details, see 
\cite{Blows-Lloyd,Lloyd}.
\medskip

The coefficients $\eta_{2k}$ are known as \emph{focal values}. From Lyapunov theorems \cite{Hahn,Wig}, the stability of the origin depends on the first nonzero focal
value, and the origin is a center if and only if all the focal values are zero. Since the focal values $\eta_{2k}$ are  all polynomials, Hilbert's Nullstellensatz \cite{Almira,Hilbert} implies that the ideal spanned by them has a finite basis. Therefore, there exists $M>0$ such that $\eta_{2\ell}=0,$ for $\ell\leq M$, implies $\eta_{2\ell}=0$
 for all $\ell$. The value of $M$ is not known \emph{a priori}, so the number of focal values one needs to calculate for finding the basis is also unknown. Due to the extension of the calculations one needs to perform in order to compute the focal values, the use of computational packages becomes handy. For instance, the software
 \emph{Mathematica} \cite{mat} can be used to calculate the first focal values, using the relations  $\eta_2=\eta_4=\dots=\eta_{2k}=0$ 
to eliminate some of the 
variables in $\eta_{2k+2}$. Removing common factors from the reduced focal values one proceeds until it is possible to prove that the reminding expressions do
 not vanish simultaneously. The conditions over the parameters of $X$ used for computing the focal values provide the complete necessary center conditions. To prove 
that the conditions obtained are also sufficient one can use, for instance, symmetries or integrating factors for the vector field~$X$. The reduced focal value $\eta_{2k+2}$, with the strictly positive factors removed, is known as the \emph{Lyapunov quantity} $L(k)$. Note that $L(0)=\lambda$ and, with this notation, the necessary center conditions are given by a basis for 
$\{ L(k)=0, \ \forall k\ge 0\}$.  In this context, we say that the origin is a \emph{weak focus of order} $k$ if $L(i)=0$, for $i=0,1,2,\ldots, k-1$ and $L(k) \ne 0$.\\

From Lyapunov's theorems \cite{Hahn,Wig} follows that a weak focus of order $k$ is attracting if $L(k)<0$ and repelling if $L(k)>0$. In addition, if the origin is a weak focus of order $k$, at most $k$ limit cycles can bifurcate from the origin \cite{Blows-Lloyd}. These limit cycles arise from Hopf bifurcations and are known as \emph{small-amplitude} or 
\emph{infinitesimal} limit cycles.

\section{Perturbation of Hamiltonian systems and Melnikov functions}
\label{sec:Melnikov}

From section \ref{sec:center} we know that one way to obtain limit cycles in planar polynomial systems and, in particular, in a Kukles system, is via changes of stability of a saddle-focus equilibrium point via Hopf bifurcations; the use of Lyapunov quantities plays a relevant role.  In this section we consider limit cycles that bifurcate from a center, via the perturbation of a Hamiltonian system. We consider a planar vector field $X_{\eps}$ of the form
\begin{equation}\label{eq:Pert}
X_{\eps}=
\left\{
\begin{array}{rcl}
\dot x &=& \ds\frac{\partial H}{\partial y}(x,y)+\ds\sum_{k=1}^{\infty}\eps^kf_k(x,y, \delta)\\[0,3cm]
\dot y &=& -\ds\frac{\partial H}{\partial x}(x,y)+\ds\sum_{k=1}^{\infty}\eps^kg_k(x,y,\delta)
\end{array}
\right.,
\end{equation}
where $H, f_k,g_k$ are analytic functions in an open set $U\subset \R^2, \ \eps$ is a small non-negative parameter and $\delta \in D\subset \R^n$ is a vector of parameters in a compact set $D$. We assume that the origin is a singularity of \eqref{eq:Pert}
 for all $\eps \in \mathbb R$ and the system has a center at the origin for $\eps=0$. We call \eqref{eq:Pert}, with $\eps=0$ the unpertutbed system, and for $\eps \ne 0$  perturbed one. The unperturbed system $X_0$ is Hamiltonian, and has the form
\begin{equation}\label{eq:Nopert}
X_0=
\left\{
\begin{array}{rcl}
\dot x &=& \ds\frac{\partial H}{\partial y}(x,y)\\[0,3cm]
\dot y &=& -\ds\frac{\partial H}{\partial x}(x,y)
\end{array}
\right.
\end{equation}

For $\eps_0>0$, we assume the existence of an open interval $J=(0,a)$ on the $x$-axis, transverse to the flow of $X_\eps$ for $|\eps|\le\eps_0$. There exists a subinterval $I\subset J$ such that the Poincar\'e return map 
$$
\begin{array}{rcl}
\pi:I\times (-\eps_0, \eps_0) & \mapsto & J\\[0,3cm]
(x,\eps) & \mapsto & \pi(x,\eps)
\end{array}
$$
is well defined, and maps $(x,\eps)$ to the $x$-coordinate $\pi(x,\eps)$ of the first return of the point $(x,0)$ to $J$ by the flow of $X_\eps$. We can define the displacement function
$$
\begin{array}{rcl}
d:I\times (-\eps_0, \eps_0) & \mapsto & \R\\[0,3cm]
(x,\eps) & \mapsto & d(x,\eps)
\end{array}
$$
by $d(x,\eps)=\pi(x,\eps)-x$.
 \medskip

An orbit of $X_\eps$ through the point $(x,0)$, with $x\ne0$ is periodic if and only if $(x,\eps)$ is as zero of the displacement function $d$. Since the origin is a center for $\eps=0$,  we have $d(x,0)=0$ for $x\in I$. Therefore, for $x>0$  and $\eps_0$ small enough, we have the following series expansion for $d$:
\begin{equation}
d(x,\eps)\ =\ \ds\sum_{k=1}^{\infty}M_k(x)\eps^k \ = \ \left.\ds\sum_{k=1}^{\infty}\ds\frac{1}{k!}\ds\frac{\partial d(x,\eps)}{\partial \eps^k}\right|_{\eps=0}\eps^k.
\label{eq:Melf}
\end{equation}

The first nonzero function $M_n(x)$ in \eqref{eq:Melf} is known as the \emph{n-th  Melnikov function}. For $n=1$, $M_1$ is known as the \emph{Melnivok integral}. Finding simple zeroes of the first nonzero Melnikov function $M_n$ means finding limit cycles that bifurcate from perturbing the Hamiltonian system \eqref{eq:Nopert}. One usually aims for finding $M_1$, but if $M_1(x)\equiv 0$ one needs to calculate higher-order Melnikov functions. Unfortunately, there is no explicit formula in the literature for the $n$-th Melnikov function, except for special cases when the hamiltonian function asocciated to $X_0$ has a particular form \cite{Fran}, and for some recursive formulas.

\subsection{Representations for Melnikov functions and limit cycles}
Suppose that system \eqref{eq:Nopert} has a family of periodic orbits, given by  $L_h:H(x,y)=h,\ h \in (0,\beta)$ and such that $L_h$ approaches the origin when $h\to 0$. Let
$h=h_0 \in (0,\beta)$ and $A(h_0)\in L_{h_0}$. Consider the transverse section $l$ of \eqref{eq:Nopert} through $A(h_0)$. For $h$ close to $h_0$ the periodic orbit~$L_h$ intersects $l$ in a unique point $A(h)$. Consider the positive orbit $\gamma(h,\eps, \delta)$ of~\eqref{eq:Pert} with initial point $A(h)$. Let $B(h,\eps,\delta)$ be the first intersection of $\gamma$ with $l$. Then,
$$
\begin{array}{rcl}
H(B)-H(A)&=&\ds\int_{AB}dH \\[0,3cm]
&=& \eps[M(h,\delta)+\mathcal O(\eps)]\\[0,3cm]
&=& \eps F(h,\eps,\delta)
\end{array}
$$
with
$$
\begin{array}{rcl}
M(h,\delta)&=&\ds\oint_{L_h}(H_yq+H_xp)dt\\[0,3cm]
&=&\ds\oint_{L_h}(qdx-pdy)dt\\[0,3cm]
&=&\ds\iint_{H\le h}(p_x+q_y)dxdy.
\end{array}
$$
For $\eps$ small enough, system \eqref{eq:Pert} has a limit cycle near the origin if and only if $F(h,\eps,\delta)$ has a positive zero for $h$ near $h=0$. From here, $M(h,\delta)$ controls the number of limit cycles of \eqref{eq:Pert} near the origin. We can assume that the Hamiltonian function has the form
\begin{equation}\label{eq:FormaHam}
H(x,y)=\ds\frac\omega 2(x^2+y^2)+\ds\sum_{i+j\ge3}h_{ij}x^iy^j,\ \omega>0.
\end{equation}
The following theorem holds \cite{Han}:
\begin{theorem}\label{theo:teoHan}
Let H be of the form \eqref{eq:FormaHam}. Then, $M(h,\delta)$ is $C^{\infty}$ in $0\le h \ll 1$, with
$$
M(h,\delta)=h\ds\sum_{s\ge 0}b_s(\delta)h^s
$$
If \eqref{eq:Nopert} is analytic, so is $M$. Moreover, if there exists $k\ge1,\delta_0 \in D$ such that $b_k(\delta_0)\ne0$ and
\begin{equation}
\begin{array}{ccc}
&b_j(\delta_0)=0, \qquad j=0,1,\ldots,k-1,&\\[0,3cm]
&det\left(\ds\frac{\partial (b_0,\ldots,b_{k-1})}{\partial (\delta_1,\ldots,\delta_k)}\right)(\delta_0)\ne0,&
\end{array}
\label{eq:Handet}
\end{equation}
where $\delta=(\delta_1,\ldots,\delta_m), \ m\ge k,$ then there exist a constant $\eps_0>0$ and a neighborhood $V$ of the origin such that for all $0<|\eps|<\eps_0$ and $|\delta-\delta_0|<\eps_0$, \eqref{eq:Pert} has at most $k$ limit cycles in $V$. Moreover, for any neighborhood $V_1$ of the origin there exists $(\eps,\delta)$ near $(0,\delta_0)$ such that system \eqref{eq:Pert} has $k$ limit cycles in $V_1$.
\end{theorem}
This theorem provides an analytic expression for the first Melnikov function and states that the number of limit cycles that bifurcate from a center is related to the algebraic multiplicity of $0$ as a root of the first Melnikov function.
\bigskip

Between 1996 and 2002 Giacomini, Llibre and Viano worked on a method for computing analytically the shape of the limit cycles that bifurcate from a center. The method is based on the following theorem \cite{Via-Lli-Gia2}:
\begin{theorem}
Let $X=(P,Q)$ be a $C^1$ vector field on an open set $U\subset \R^2$. Let $V=V(x,y)$ be a $C^1$ solution of the partial differential equation
\begin{equation}\label{eq:Recfi}
P\ds\frac{\partial V}{\partial x}+Q\ds\frac{\partial V}{\partial y}=\left(\ds\frac{\partial P}{\partial x}+\ds\frac{\partial Q}{\partial y}\right)V.
\end{equation}
If $\gamma$ is a limit cycle of $X$, then $\gamma \subset \left\{ (x,y)\in U : V(x,y)=0\right\}.$
\end{theorem}

The function $V$ is known as a \emph{reciprocal integrating factor}, name after the fact that the function $R=\frac1V$ defined on $U-\{V=0\}$ is an integrating factor of the vector field $X$. In \cite{Via-Lli-Gia}, the authors use the function $V$ in order to describe the limit cycles that bifurcate from a Hamiltonian center. In \cite{Via-Lli-Gia3}, they  look for a representation for  $V$ when the vector field $(P,Q)$ is of the form \eqref{eq:Pert}. It is proved that an analytic solution for \eqref{eq:Recfi} has the form
$$V=V(x,y,\eps)=\ds\sum_{k=0}^{\infty}\eps^k V_k(x,y),$$
and obtain a recursive formula for $V_k$ by solving a sequence of linear partial differential equations. For $V_0$ the equation is
$$\ds\frac{\partial H}{\partial y}\frac{\partial V_0}{\partial x}-\ds\frac{\partial H}{\partial x}\frac{\partial V_0}{\partial y}=0$$
while for $V_k$ the equation is 
$$
\ds\frac{\partial H}{\partial y}\frac{\partial V_k}{\partial x}-\ds\frac{\partial H}{\partial x}\frac{\partial V_k}{\partial y}=\ds\sum_{j=0}^{k-1}V_j\left(\ds\frac{\partial f_{k-j}}{\partial x}+\ds\frac{\partial g_{k-j}}{\partial y}\right)-\ds\sum_{j=0}^{k-1}\left( f_{k-j}\ds\frac{\partial V_j}{\partial x} +g_{k-j}\ds\frac{\partial V_j}{\partial y} \right).
$$
The solution for  $V_0$ is used for finding $V_1$, and so on. One of the main results of~\cite{Via-Lli-Gia3} is the following:
\begin{theorem}\label{teo:LliMel}(Giacomini, Llibre, Viano, \cite{Via-Lli-Gia3})
\begin{enumerate}
\item[(i)] If  $T=T(h)$ is the period of the periodic orbit $H(x,y)=h$ of \eqref{eq:Nopert}, then
$$V_0(h)=\ds\int_{0}^{T(h)}\left(f_1\ds\frac{\partial H}{\partial x}+g_1\ds\frac{\partial H}{\partial y}\right)dt$$
is the Poincaré-Melnikov integral of \eqref{eq:Pert} associated to the orbit $H(x,y)=h$ when $V_0(h)\not \equiv 0$.
\item[(ii)] $V_k$, with $k\ge 1$ is calculated recursively as $V_k(x,y)=V_k^p(x,y)+W_k(h)$, where
$$
\begin{array}{rcl}
V_{k+1}^p(x,y)&=&V_0(h)\ds\int_0^t\left(\ds\frac{\partial f_{k+1}}{\partial x}+\ds\frac{\partial g_{k+1}}{\partial y}\right)dt-V_0'(h)\ds\int_0^t\left(f_{k+1}\ds\frac{\partial H}{\partial x}+g_{k+1}\ds\frac{\partial H}{\partial y}\right)dt\\[0,3cm]
&&+\ds\sum_{j=1}^k\ds\int_0^t\left[ V_j\left( \ds\frac{\partial f_{k+1-j}}{\partial x}+\ds\frac{\partial g_{k+1-j}}{\partial y}\right)-\left(f_{k+1-j}\ds\frac{\partial V_j}{\partial x}+g_{k+1-j}\ds\frac{\partial V_j}{\partial y}\right)\right]dt
\end{array}
$$
and $V_0(h)W_k'(h)-V_0'(h)W_k(h)=\xi (h)$, with
$$
\begin{array}{rcl}
\xi (h)&=& V_0(h)\varphi_{k+1}'(h)-V_0'(h)\varphi_{k+1}(h)+\ds\sum_{j=1}^{k-1}\left[W_j(h)\varphi _{k+1-j}'(h)-W_j'(h)\varphi_{k+1-j}(h) \right]\\[0,3cm]
&& + \ds\int_0^{T(h)}\ds\sum_{j=1}^k\left[ V_j^p\left( \ds\frac{\partial f_{k+1-j}}{\partial x}+\ds\frac{\partial g_{k+1-j}}{\partial y}\right)-\left(f_{k+1-j}\ds\frac{\partial V_j^p}{\partial x}+g_{k+1-j}\ds\frac{\partial V_j^p}{\partial y}\right)\right]dt.
\end{array}
$$
The symbol $'$ denotes derivation with respect to $h$, and
$$\varphi_k(h)=\ds\int_0^{T(h)}\left(f_k \ds\frac{\partial H}{\partial x}+g_k\ds\frac{\partial H}{\partial y}\right)dt.$$
\end{enumerate}
\end{theorem}
From part (i) of the theorem above one can compute the first Melnikov function of a perturbed Hamiltonian system, while part (ii) is for computing the following $V_k$, which determine the shape and position of the limit cycles that bifurcate from a center. In \cite{Via-Lli-Gia3} there are examples showing explicit ecuations for the limit cycles that bifurcate in the Van der Pol and Liénard systems.
\medskip

Poincaré proved that the maximum number of simple zeroes of $V_0(h)$ coincide with the maximum number of limit cycles of system \eqref{eq:Pert} that bifurcate from periodic orbits $H(x,y)=h$ of \eqref{eq:Nopert} when $\eps$ is small. This limit cycles are known as  \emph{global} or \emph{large-amplitude} limit cycles.
\medskip

When the first Melnikov function is identically null, one needs to compute higher order Melnikov functions, which is a challenging task. There are some results for second order Melnikov due to Iliev \cite{Il}, and for higher order Melnikov funcions via Abelian integrals \cite{Fran}. Some explicit examples in wich the vector field is written in complex coordinates can be found in
\cite{Bui-Gas-Yang} and \cite{Fran}.

Overall, in one hand we have small-amplitude limit cycles bifurcating from an equilibrium point, while on the other hand we have large-amplitude limit cycles bifurcating from a Hamiltonian center. These two mechanisms for obtainig limit cycles and the limit cycles themselves are different in principle. However, they can be related. To our understanding, the relation between these two type of limit cycles, conditions for which they can coincide and what type of vector fields allow this to happen is far from being completely understood and remains an interesting challenge beyond the scope of this paper; we show an example of this relation in section \ref{sec:odd}.

\section{Coexistence of limit cycles with an invariant circle in Kukles systems}
\label{sec:kukcir}

In this section we study a family of Kukles systems, in which there is an invariant circle that coexists with an unique limit cycle. We prove this via Lyapunov quantities and second-order Melnikov functions.

\medskip

We consider the vector field family given by:

\begin{equation}\label{eq:lineal}
X_{\mu}:\left\{
\begin{array}{rcl}
\dot x &=& -y \\
\dot y &=& x+y(x^2+y^2-1)(ax+by+c),
\end{array}\right.
\end{equation}
where $\mu=(a,b,c)\in \R^3$. We prove that the algebraic curve $\mathcal C(x,y)=x^2+y^2-1$ is invariant for the system, as well as an algebraic limit cycle. We prove that the limit cycle is unique and can be obtained as a small-amplitude limit cycle that bifurcate from a nonhyperbolic focus at the origin. We begin with the following:

\begin{lemma}
For all $\mu \in \R^3$, the circle $\mathcal C=0$ is an algebraic invariant curve of~\eqref{eq:lineal}.
\end{lemma}
\begin{proof}
We note that, for all $\mu \in \R^3$, the flow along the curve satisfies
$$\dot {\mathcal C}= \ds\frac{\partial \mathcal C}{\partial x}\dot x+\ds\frac{\partial \mathcal C}{\partial y}\dot y=\mathcal C(x,y)2y^2(ax+by+c).$$

Therefore, $\mathcal C=0$ is an invariant circle.
\end{proof}

Note that the curve $\mathcal C=0$ encloses the equilibrium point at the origin. In addition, there are conditions over the parameters of system \eqref{eq:lineal} for which there are no limit cycles inside the open region bounded by $\mathcal C=0$. This is shown in the following result:

\begin{lemma}
If $\mu=(0,0,c)$, with $c\ne 0$, system \eqref{eq:lineal} does not have limit cycles in the region $A=\{ (x,y)\in \R^2: x^2+y^2<1\}$.
\end{lemma}
\begin{proof}
Let $\tilde{\mu}=(0,0,c),\ c\ne 0$. The function
$$
\mathcal D(x,y)=\ds\frac{1}{x^2+y^2-1}
$$
is a Dulac function of $X_{\tilde{\mu}}$ on $A$, which implies that  \mbox{$div(\psi X_{\tilde \mu})=cte$}. From Dulac's  criterion \cite{And2} there are no periodic orbits of $X_{\tilde{\mu}}$ in $A$.
\end{proof}

In what follows we consider $a^2+b^2 \ne 0$. In this case the circle $\mathcal C=0$ is an invariant limit cycle:

\begin{theorem}\label{theo:Melnikov}
Let $c\ne 0$. The curve $\mathcal C=0$ is an algebraic limit cycle of \eqref{eq:lineal}.
\end{theorem}
\begin{proof}
We calculate the first Melnikov function for the vector field $X_\mu$. Rescaling parameters by
$$a\to \eps a;\ b\to \eps b ;\ c\to \eps c,$$
and reversing time $t\to -t $, system \eqref{eq:lineal} can be written as
\begin{equation}\label{eq:linealHam}
\begin{array}{rcl}
\dot x &=& y \\
\dot y &=& -x+\eps y(1-x^2-y^2)(ax+by+c).
\end{array}
\end{equation}
System \eqref{eq:linealHam} is a perturbation of the Hamiltonian system
$$
\begin{array}{rcl}
\dot x &=& y \\
\dot y &=& -x,
\end{array}
$$
with Hamiltonian associated function $H(x,y)=\ds\frac{x^2+y^2}{2}$. We parameterize the orbit of the isocronus center $H=h$ by
$$
\begin{array}{rcl}
x(t)&=&\sqrt{2h}\cos t\\
y(t)&=&-\sqrt{2h}\sin t
\end{array}; \ t \in [0,2\pi].
$$
From theorem \ref{teo:LliMel}, the first Melnikov function has the form
$$M_1(h)=2h(2h-1)c\pi .$$
Then, $M_1$ has as simple zero for $h=\ds\frac12$, which corresponds to the curve $\mathcal C=0$. Therefore,  the curve $\mathcal C=0$ in an algebraic limit cycle of \eqref{eq:lineal} 
\end{proof}
The limit cycle obtained above can bifurcate from the origin, as we can see in the following theorems:
\begin{theorem}
\label{theo:cl}
Let $a,b \in \R-\{0\}$. If $c=0$, system \eqref{eq:lineal} has a weak focus of order~$1$ at the origin. This limit cycle is repelling for $ab<0$ and  attracting for $ab>0$.
\end{theorem}
\begin{proof}
The linear part of \eqref{eq:lineal} is given by:
$$
DX_{\mu}(0,0)\;=\;
\left(\begin{array}{cc}
0 & -1 \\
 1 & -c \\
\end{array}\right).
$$
Note that the origin is a weak focus if $c=0$. To calculate its order, we compute the Lyapunov quantities $L(k),\ k\ge0$ in order. We have $L(0)=c=0$, and 
$$L(1)=-\ds\frac{ab}{8},$$
which proves the theorem.
\end{proof}
\begin{theorem}
\label{theo:cl2}
In parameter space ${\mathbb R}^3$, there exists an open set $\mathcal{N}$ such that, for all $\mu \in \mathcal{N}$, system \eqref{eq:lineal} has at least one small-amplitude limit cycle.
\end{theorem}
\begin{proof}
From theorem \ref{theo:cl}, system \eqref{eq:lineal} has a weak focus of order 1 at the origin for $\tilde\mu =(a,b,0)$, with $ab\ne 0$. If $ab<0$ the weak focus is repelling, then we can perturb the system  in such a way that $c<0$, changing the stability of the origin and thus creating a new repelling limit cycle via a Hopf bifurcation. Similarly, if $ab>0$, the weak focus is attracting and we can perturb the system with $c>0$, creating an attracting limit cycle.
\end{proof}
We can also show center conditions in system \eqref{eq:lineal}:
\begin{theorem}
System \eqref{eq:lineal} has a center at the origin if and only if $c=a=0$ or $c=b=0$.
\end{theorem}
\begin{proof}
From theorems \ref{theo:cl} and \ref{theo:cl2}, for $c=0$ the Lyapunov quatities are $\ L(0)=0$ and $L(1)=-\ds\frac{ab}{8}$. If $a=0$ or $b=0,\ L(1)=L(2)=L(k)=0, \ \forall k\ge0$ and we have necessary center conditions. On the other hand, for $c=a=0;\ b\ne 0$, system \eqref{eq:lineal} has the form:
$$
\begin{array}{rcl}
\dot x &=& p(x,y)\ = \ -y, \\
\dot y &=& q(x,y) \ =\ x+b(-y^2+x^2y^2+y^4).
\end{array}
$$
In this case we have the symmetries $p(x,-y)=-p(x,y)$ and $q(x,-y)=q(x,y)$, which prove that $X_\mu$ has a center at the origin if $c=a=0$.\\
Analogously, if $c=b=0;\ a\ne 0$, the vector field has the form
$$
\begin{array}{rcl}
\dot x &=&p(x,y)\ =\ - y, \\
\dot y &=& q(x,y)\ =\ x+a(-xy+x^3y+xy^3),
\end{array}
$$
and the symmetries $p(-x,y)=p(x,y);\ q(-x,y)=-q(x,y)$. Therefore, the origin is a center of $X_\mu$ for $c=b=0$.\\
This proves that the conditions are also sufficient.
\end{proof}
When the parameter $c$ vanishes in \eqref{eq:lineal}, the information obtain from the first Melnikov function is lost, since $M_1\equiv 0$ for $c=0$. However, following the method from \cite{Fran} and \cite{Il}, it is possible to compute the second Melnikov function. This is shown in the following result:
\begin{theorem}
If $ab\ne 0$, system \eqref{eq:lineal} has exactly one limit cycle.
\end{theorem}
\begin{proof}
From theorem \ref{theo:cl}, for $c=0$ there exists a limit cycle inside the region bounded by
$ \mathcal C=0$. For small values of $h$, the second Melnikov function can be represented by:
$$
M_2(h)=\ds\int_{H=h}\left[ G_{1h}(x,y)P_2(x,h)-G_1(x,y)P_{2h}(x,y) \right]dx,
$$
where
$$
\begin{array}{rcl}
G_{1h}(x,y)&=&\ds\frac{ax-ax^3-3axy^2}{y}\\[0,2cm]
P_2(x,h)&=&b2hx-4bh^2x-b\ds\frac{x^3}{3}+\ds\frac23bhx^3\\[0,2cm]
G_1(x,y)&=&-axy-ax^3y-axy^3\\[0,3cm]
P_{2h}(x,y)&=&\ds\frac{6bx+2bx^3-24bxy}{3y}
\end{array}
$$
then $M_2(h)=\ds\frac13abh(2h-1)(6-7h+28h^2)\pi$. The closed orbit associated to $h=\ds\frac12$ is still a limit cycle. Furthermore, the quadratic polynomial $6-7h+28h^2$ does ot have real roots; this holds for all $a,b \in \R-\{0\}$. Then, the limit cycle is unique.
\end{proof}

\begin{remark}
It is expected that the Melnikov function up to some order provides the total number of limit cycles, but there are cases when the Melnikov function does not control the limit cycles of small amplitude. It is worth asking about the relation between small- and large-amplitude limit cycles, or equivalently about 
the relation (if exists) between the coefficients of the Melnikov functions and the Lyapunov quantities. There is no general answer to this question, and we think an answer may be of help on tackling Hilbert's 16th problem. In \cite{Han2} the authors give a partial answer to this question, about which we mention in next section.
\end{remark}
\section{Coexistence of limit cycles with an invariant circle in a family of Kukles systems with arbitrary odd degree}
\label{sec:odd}

In this section we study the coexistence of limit cycles with an invariant circle in a family of Kukles systems of arbitrary odd degree. We study Lyapunov quantities and the first Melnikov function, and show a new concrete example, different from the one given in \cite{Han}, in which theorem \ref{theo:teoHan}  does not apply, i.e. it is not possible to stablish the cyclicity of the system with only the coefficients of the first Melnikov function.
\medskip

We consider the following family of Kukles systems:

\begin{equation}\label{eq:impar}
X=\left\{ \begin{array}{rcl}
\dot x &=&-y,\\
\dot y &=&x+y(1-x^2-y^2)(b_{00}+\displaystyle\sum_{i+j=1}^{n}b_ {2i2j}x^{2i}y^{2j}),
\end{array}\right.
\end{equation}
where $b_{2i 2j} \in \R$ and $i,j$ are non-negative integer numbers. This family of planar vector fields of degree $2n+3$ has the invariant circle $C(x,y)=0$,  with $$C(x,y)=1-x^2-y^2,$$
as we show in the following lemma:
\begin{lemma}
The circle $C=0$ is an algebraic invariant curce of \eqref{eq:impar}.
\end{lemma}
\begin{proof}
A straightforward calculation shows that
$$\displaystyle\frac{\partial C}{\partial x}\dot x+\displaystyle\frac{\partial C}{\partial y}\dot y=C(x,y)(-2y^2)(b_{00}+\displaystyle\sum_{i+j=1}^{n}b_{2i\;2j}x^{2i}y^{2j}),
$$
which implies that $C$ is invariant.
\end{proof}
The origin is the only singularity of system \eqref{eq:impar}, and it is a \emph{center-focus} for $b_{00}=0$. This can be easily checked by looking at the linearization  of the system, given by
$$
DX(0,0)\;=\;
\left(\begin{array}{cc}
0 & -1 \\
 1 & b_{00} \\
\end{array}\right)
$$

We study the coexistence of limit cycles of  \eqref{eq:impar} with the invariant curve $C=0$. To this end, we calculate the first Melnikov function. We have the following theorem:
\begin{theorem}
\label{theo:Melimpar}
For all $n \in \N$, system \eqref{eq:impar} has at most $n+1$ limit cycles of large amplitude.
\end{theorem}
\begin{proof}
We write system \eqref{eq:impar} as the perturbation of a Hamiltonian system.  The parameter rescaling $b_{2i2j}\to\epsilon \cdot b_{2i,2j}$, for $\epsilon>0$, gives the vector field
\begin{equation}\label{eq:imparpert}
X_{\epsilon}=\left\{ \begin{array}{rcl}
\dot x &=&-y,\\
\dot y &=&x+\epsilon y(1-x^2-y^2)(b_{00}+ \displaystyle\sum_{i+j=1}^{n}b_{2i\;2j}x^{2i}y^{2j}).
\end{array}\right.
\end{equation}
Note that the unperturbed system $(\epsilon =0)$ is Hamiltonian:
\begin{equation}\label{eq:imparHam}
X_{0}=\left\{ \begin{array}{rcl}
\dot x &=&-y,\\
\dot y &=&x,
\end{array}\right.
\end{equation}
with associated Hamiltonian function $H(x,y)=\ds\frac{x^2+y^2}{2}$. Reverting the direction of the flow via the time rescaling $t\to -t$, we finally write the system as
$$
X_{\epsilon}=\left\{ \begin{array}{rcl}
\dot x &=&\displaystyle\frac{\partial H}{\partial y},\\[0,4cm]
\dot y &=&-\displaystyle\frac{\partial H}{\partial x}+\epsilon  g_1(x,y),
\end{array}\right.
$$
where $$g_1(x,y)=y(x^2+y^2-1)(b_{00}+ \displaystyle\sum_{i+j=0}^{n}b_{2i\;2j}x^{2i}y^{2j}).$$
We now consider the period $T=2\pi$ of the orbit $H(x,y)=h$ of the unperturbed system.  We parameterize the orbit by
$$
\begin{array}{rcl}
x(t)&=&\sqrt{2h}\cos (t),\\[0,3cm]
y(t)&=&-\sqrt{2h}\sin (t),
\end{array}
$$
with $h>0$. From theorem \ref{teo:LliMel}, the first Melnikov function has the form
$$
\begin{array}{rcl}
M_1(h)&=&\displaystyle\int_{0}^{T}y(t)g_1(x(t),y(t))dt\\[0,5cm]
&=&2h(2h-1) \left(\pi b_{00}+\displaystyle\sum_{i+j=1}^{n}b_{2i,2j}(2h)^{i+j}\displaystyle\int_{0}^{2\pi}\cos^{2i}(t) \sin^{2j+2}(t)dt\right).
\end{array}
$$
Note that
$$
\displaystyle\int_{0}^{2\pi}\cos^{2i}(t) \sin^{2j+2}(t)dt=4\displaystyle\int_{0}^{\frac\pi 2}\cos^{2i}(t) \sin^{2j+2}(t)dt,
$$
and for $r,s>0$
$$
\displaystyle\int_{0}^{\frac\pi 2}\cos^{2r-1}(t) \sin^{2s-1}(t)dt=\displaystyle\frac{\Gamma(r)\Gamma(s)}{2\Gamma(r+s)},
$$
where $\Gamma(\cdot)$ corresponds to the \emph{Gamma} function. Then
$$
\begin{array}{rcl}
M_1(h)&=&h(2h-1)\left(2b_{00}+ \displaystyle\sum_{k=1}^n\sum_{i+j=k}\displaystyle\frac{(2i)!(2j)!(2j+1)}{2^{i+j-1}i!j!(i+j+1)!}b_{2i 2j}h^{i+j}\right)\pi.
\end{array}
$$
This polynomial has degree $n+2$.\\
Now note that $h=\frac12$ is a simple zero of $V_0(h)$, i.e. the orbit $H(x,y)=\frac12$ is a limit cycle or equivalently  the curve $C=0$ is a limit cycle. The simple zero $h=0$ does not provide any limit cycle since the orbit $H=0$ reduces to the origin. Therefore, system \eqref{eq:impar} has at most  $n+1$ large-amplitude limit cycles.
\end{proof}
\medskip

\begin{example}
\label{ex:ejem}
System \eqref{eq:impar} with $n=3$ gives:
\begin{equation}\label{eq:ejem}
X=\left\{ \begin{array}{rcl}
\dot x &=&-y,\\[0,1cm]
\dot y &=&x+y(1-x^2-y^2)(b_{00}+b_{20}x^2+b_{02}y^2+b_{40}x^4+b_{22}x^2y^2+\\[0,1cm]&&
b_{04}y^4+b_{60}x^6+b_{42}x^4y^2+b_{24}x^2y^4+b_{06}y^6),
\end{array}\right.
\end{equation}
which has degree 9. Using theorem \ref{theo:Melimpar}, the first Melnikov function has the form
$$
\begin{array}{rcl}
M_1(h)&=&\displaystyle\frac \pi4h(2h-1)(8b_{00}+4(3b_{02}+b_{20})h+4(5b_{04}+b_{22}+b_{40})h^2+\\[0,3cm]&&
(35b_{06}+5b_{24}+3b_{42}+5b_{60})h^3).
\end{array}
$$
From here, at most 4 limit cycles can bifurcate from the Hamiltonian center, including the invariant circle $C=0$.\\

Using the software \emph{Mathematica} we calculate the Lyapunov quantities $L(k)$, $k\ge 0$. We have:
$$L(0)=b_{00},$$
and if  $b_{00}=0,\ L(0)=0$. Then
$$L(1)=\displaystyle\frac{3b_{02}+b_{20}}{8}.$$
If $b_{20}=-3b_{02},\ L(1)=0$, then we have
$$L(2)=\displaystyle\frac{5b_{04}+b_{22}+b_{40}}{16},$$
and if $b_{22}=-5b_{04}-b_{40},\ L(2)=0$, then
$$L(3)=\displaystyle\frac{-6b_{02}^3+35b_{06}+5b_{24}+3b_{42}+5b_{60}}{128}.$$
We assume the condition 

\begin{equation}
5b_{24}+3b_{42}+5b_{60}\ne0,
\label{eq:ast}
\end{equation}
so that the higher order term in $M_1$ is not zero. If we set
\mbox{$b_{42}=\displaystyle\frac{6b_{02}^3-35b_{06}-5b_{24}-5b_{60}}{3}$}, $L(3)=0$, we have
$$L(4)=\displaystyle\frac{3}{128}b_{02}^2(2b_{02}-b_{04}+b_{40}).$$
In addition, if $b_{04}=2b_{02}+b_{40}$ we have that $L(4)=0$. It follows that
$$
L(5)=\displaystyle\frac{b_{02}(144b_{02}^2+9b_{02}b_{24}+204b_{02}b_{40}+80b_{40}^2+21b_{02}b_{60})}{1536},$$
Now, if $b_{40}=0$, we can write
$$\ L(5)=\ds\frac{3}{1536}b_{02}^2(48b_{02}+3b_{24}+7b_{60}).$$
From here, setting $b_{24}=-\ds\frac13(48b_{02}+7b_{60})$, we have $L(5)=0$ and then
$$
L(6)=-2b_{02}^2(246b_{02}+9b_{02}^3-183b_{06}-5b_{60}).
$$
For $b_{60}=\ds\frac35(82b_{02}+3b_{02}^2-61b_{06})$, it follows that $L(6)=0$ and we can write
$$
L(7)=\ds\frac{18}{25}b_{02}(749152b_{02}^2+77586b_{02}^4+1737b_{02}^6-
785952b_{02}b_{06}-45918b_{02}^3b_{06}+171488b_{06}^2)
$$
For this set of conditions upon the parameters, condition \eqref{eq:ast} is equivalent to $b_{02}\ne0$. If $b_{02}=0$, system \eqref{eq:ejem} has a center at the origin. Let's consider $b_{02} \ne 0$ and define functions $\xi$ and $\zeta$ defined over an open interval in $\mathbb R$ by
$$
\begin{array}{rcl}
\xi (b_{02})&=& 392976b_{02}+22959b_{02}^3-5\sqrt{1038382336b_{02}^2+819584160b_{02}^4+916941b_{02}^6},\\[0,3cm] 
\zeta(b_{02})&=& 392976b_{02}+22959b_{02}^3+5\sqrt{1038382336b_{02}^2+819584160b_{02}^4+916941b_{02}^6}.
\end{array}
$$
We can now write the following Lyapunov quantities by
$$\begin{array}{rcl}
L(7)&=&b_{02}(b_{06}-\xi)(b_{06}-\zeta),\\[0,3cm]
L(8)&=&F(b_{02},b_{06}),
\end{array}
$$
where $F:A\subset \mathbb{R}^2 \to \mathbb R$ and $A$ is an open set. On the curve $b_{06}=\xi(b_{02})$, we have that $L(7)=0$. The implicit function theorem implies the existence of a point $P=(b_{02}^{\ast},b_{06}^{\ast})$ on the graph of $\xi$ on the $(b_{02},b_{06})$-parameter plane, and a neighborhood $\mathcal U$ of $P$ where it is possible to define a real function $\gamma$ satisfying $b_{06}=\gamma(b_{02})$, with $\gamma'>0, \; \gamma''<0$ and with $F(b_{02},\gamma(b_{02}))=0$, such that $L(9)>0$ in $\mathcal U$. Similarly, on the curve $b_{06}=\zeta(b_{02})$, $L(7)$ also vanishes and in a similar way we can use the implicit function theorem to find a point $Q(\tilde b_{02},\tilde b_{06})$ on the graph of $\zeta$ and a neighborhood $\mathcal V$ of $Q$ when it is possible to define a curve $b_{06}=\phi(b_{02})$ with $\phi'>0,\; \phi''>0$ y and with $F(b_{02},\phi(b_{02}))=0$, such that $L(9)< 0$ in $\mathcal V$; we omit the exact expressions of $L(8)$ and $L(9)$ due to their extension. From here, at least 9 limit cycles can bifurcate from the origin.\\
\end{example}

Note that the Lyapunov quantities $L(0),L(1)$ and $L(2)$ appear in the coefficients of the first Melnikvo funciton $M_1(h)$. This suggests a relation between the small- and large-amplitude limit cycles of system \eqref{eq:ejem}. In fact, this also suggests that the limit cycles obtained perturbing the Lyapunov quatities are related with the ones guaranteed by theorem \ref{theo:teoHan}, but there are more limit cycles that are not controled by the first Melnikov function. In this case the condition \eqref{eq:Handet} over the determinant is not satisfied. Example \ref{ex:ejem} provides a new system in which the number of limit cycles that bifurcate from a weak focus at the origin is larger than the number of limit cycles that can be obtained as simple zeros of the first Melnikov function. The fact that the Lyapunov quantities appear in the coefficients of the first Melnikov function gives a hint of that this relation may imply that some of the limit cycles obtained by the two different methods (perturbation from an equilibrium point vs perturbation from a family of periodic orbits) can coincide. The actual nature of this relation and a complete caracterization remains an interesting challenge for future research.\\

\section{Aknowledgements }

The author would like to express his gratitude to Iv\'an Sz\'ant\'o for his guidance, support and patience.
\bibliographystyle{amsplain}

\end{document}